\documentclass{amsart} 
 \usepackage{hyperref}   
 \hypersetup{
colorlinks = true,
} 
\usepackage{amsmath, amssymb, mathrsfs, amsfonts,  amsthm} 
\usepackage [all]{xy}

\DeclareMathOperator{\rank}{rank}
\DeclareMathOperator{\karea}{K-\text{area}}

\DeclareMathOperator{\Vol}{Vol}

\newtheorem {theorem} {Theorem} [section]

\newtheorem{lemma}[theorem] {Lemma}

\newtheorem {definition} [theorem] {Definition}

\newtheorem {remark} [theorem] {Remark}
\newtheorem {rem/q}[theorem]{Remark/Question}
\numberwithin {equation} {section}

\begin{document}
\author {Yasha Savelyev}
\address{Centre de Recherches Math\'ematiques, Universit\'e de Montr\'eal, C.P. 6128, Succ. Centre-ville, Montr\'eal H3C 3J7, Qu\'ebec, Canada}
\email{savelyev@crm.umontreal.ca}
\title   {Gromov $K$-area and jumping curves in $ \mathbb{CP} ^{n}$}  
\maketitle
\begin{abstract} We give here some extensions of Gromov's and Polterovich's
theorems on $\karea$ of $ \mathbb{CP} ^{n}$, particularly in the symplectic and
Hamiltonian context. Our main methods involve Gromov-Witten theory, and some
connections with Bott periodicity, and loop groups. The argument is closely
connected with study of jumping curves in $ \mathbb{CP} ^{n}$, 
and as an upshot we  prove a  new symplectic geometric  theorem on these
jumping curves.
\end{abstract}
\section {Introduction}
In \cite{Gromov} Gromov proposed an interesting way to probe the macroscopic
geometry of Riemannian and symplectic manifolds by means of the geometry of
complex vector bundles on the manifold. This is a geometric analogue
of   K-theory which is partly why it was named 
$\karea$. This
construction involves minimizing sup norm of the curvature  over all
 homologically essential vector bundles and all connections, and his main
 theorem in the Riemannian setting for spin, positively curved manifolds 
 $X,g$ used somewhat mysteriously the index theorem for the
twisted Dirac operator.

To pass to the symplectic world Gromov  considered an additional variation
over all compatible metrics, i.e. all compatible almost complex structures with
a symplectic form $\omega$ on $X$. At the moment the resulting invariant is still
very poorly understood. Here we focus on $X = \mathbb{CP} ^{n}$ and relate this
notion to the classical theory of jumping curves and  quantum classes
originally defined by the author in \cite{GS}. Interestingly, this allows us
to arrive at a purely algebraic geometric theorem in the theory of jumping
curves, as well as its symplectic generalization.

Some of the symplectic methods of the present paper continue in the spirit of
Polterovich \cite{polterovich}, and Entov
\cite{Entov}.
%
\subsection{}
We begin by discussing $\karea$ in the symplectic context. 
Fix $$p: E \to X$$  a rank $r$ complex vector bundle
with $c_1 (E)=0$ over a closed symplectic manifold $(X ^{2n}, \omega)$ and let
$p: P \to X$  denote its projectivization. Let $ \mathcal {A}$ be a Hamiltonian
connection on $P$, with curvature 2-form $R ^{ \mathcal {A}}$, which at $x \in
X$ takes values in the Lie algebra of $ \text {Ham}( p ^{-1} (x))$, i.e. the space of normalized smooth 
functions $G$ on $p ^{-1} (x)$. Where normalized means $$\int _{p
^{-1} (x)} G \Vol _{\omega _{st}} =0,$$ and $\omega _{st}$ on $p ^{-1} (x) \simeq \mathbb{CP} ^{r-1}$ is always assumed
to be the standard form with $\omega _{st} ( [line])=1$.

 Let $j_X$ be an
$\omega$-compatible almost complex structure on $X$,
 and $g _{j_X}$ the associated metric. We define the norm of curvature by
\begin{equation} \label {eq.norm} ||R ^{ \mathcal {A}}|| _{g _{j_X}} = \sup_ {x
\in X, \xi, \eta \in T_x X}  |R ^{ \mathcal {A}} (\xi, \eta)| ^{+}_H,
\end{equation}
where $| \cdot| ^{+}_H$ is ``half'' the Hofer norm $|G| ^{+}_H= \max G$, for $G:
(p^{-1} (x)
\simeq 
\mathbb{CP} ^{r-1}) \to \mathbb{R}$ in the Lie algebra of $\text{Ham}(p ^{-1}
(x))$, and the supremum is over all orthonormal pairs $\xi, \eta$. 
Here is the basic  quantity we will study:
\begin{equation} \label {eq.karea1}\karea ^{-1} (X, \omega) = \inf _{E, 
\mathcal {A}, j_X} ||R ^{  \mathcal {A}}|| _{g _{j_X}},
\end{equation}
where the infimum is over all $E$ with some non-vanishing Chern number,
and $2r  \geq \dim _{ \mathbb{R}}X$.
\begin{remark} The quantity \eqref{eq.karea1} is closely
related to the one studied by Gromov in \cite{Gromov}. This relationship is
discussed in \cite{polterovich}. We make a few comments: Gromov does not
projectivize and works with unitary connections and consequently with standard norm on the Lie algebra of $U (r)$,
he also works with the inverse of our quantity, which we symbolize
 by the superscript $-1$ in $\karea ^{-1}$. The condition that
$2r  \geq \dim _{ \mathbb{R}}X$ is related to stability for homotopy
 groups of $SU (n)$ and is vacuous if we restrict to unitary connections, since
after stabilizing $E$ we may extend the unitary connection to the stabilization,
without affecting the norm \eqref{eq.norm}. 
\end{remark}
Here is our first theorem:
\begin{theorem}  \label{thm.projective}
\begin{equation}  \karea ^{-1} (\mathbb{CP} ^{n}, E) \equiv \inf
_{\mathcal {A}, j _{X}} ||R ^{  \mathcal {A}}|| _{g _ {{j} _{X}}}  \geq 1,
\end{equation}
where the infimum is over all Hamiltonian connections $ \mathcal {A}$ on the
projectivization of a fixed complex vector bundle $E$ on $ \mathbb{CP}
^{n}$, provided that $$\rank _{ \mathbb{C}} E \geq n,$$ $c_1(E)=0$ and some
other Chern class of $E$ does not vanish.  
In particular $$\karea ^{-1} ( \mathbb{CP} ^{n}) \geq 1.$$
\end {theorem}
This answers a question of Polterovich in
\cite{polterovich} about finding bounds for $\karea ^{-1} ( \mathbb{CP} ^{2},
E)$, for rank $2$ complex vector bundle over $ \mathbb{CP} ^{2}$, with
$c_1(E)=0, c_2(E)=2$. If  $c_1$ does not vanish then the argument is more
elementary, and was already worked out in \cite{polterovich}, although one
can also adapt our discussion to subsume this case. 
 The above theorem extends:
  \begin{theorem} [Gromov,
\cite{Gromov}] \label{thm.gromov} \begin{equation*}  \karea_U ^{-1} (
\mathbb{CP}^{n}, \omega _{st}) \geq 1.
\end{equation*}
\end{theorem}
Here $U$ in $\karea_U ^{-1}$ emphasizes that Gromov worked with unitary
connections. Notice we have the same lower bound in the unitary and
Hamiltonian case. 

\begin{rem/q} We can of course define $\karea ^{-1} ( \mathbb{CP} ^{n}, g)$ for
an arbitrary non almost Kahler metric $g$ on $ \mathbb{CP} ^{n}$. 
And it is easy to find $g$ for which $$\karea ^{-1} ( \mathbb{CP} ^{n}, g) <
1,$$ as it is elementary to check: $$\karea ^{-1} ( \mathbb{CP} ^{n}, E, c \cdot
g) = 1/c\karea ^{-1} ( \mathbb{CP} ^{n}, E,  g),$$ for $c>0$. It appears
much more difficult to construct  an example of such
a $g$ with the same volume as $g _{j} = (\omega _{st}, j)$.
\end{rem/q}

 \subsection {Jumping curves in $ \mathbb{CP} ^{n}$} Although the proof of
Theorem \ref{thm.projective} is via transcendental methods of Gromov-Witten
theory, it is also closely related  to the classical notion of jumping curves in
$ \mathbb{CP} ^{n}$. In fact, as a corollary, we obtain an interesting phenomenon regarding these jumping curves. Here is a simplified
 version of the definition,  suitable in our context.
 \begin{definition} Let $E \to \mathbb{CP} ^{n}$ be a $rank _{ \mathbb{C}} E =r$
holomorphic vector bundle, with $c_1 (E)=0$. A smooth rational curve $C$ in $
\mathbb{CP} ^{n}$ will be called a \emph { \textbf{jumping curve}}, if the
restriction of $E$ to $C$ (by which we mean pullback) is not trivial as a
holomorphic vector bundle. (Actually we will just be concerned with jumping
lines.)
\end{definition}
Jumping curves $C$ can be further classified by the holomorphic isomorphism type
of $E| _{C}$, which by Grothendieck-Birkhoff theorem is:
\begin{equation} E| _{C} \simeq \bigoplus _{i} \mathcal {O} (\alpha (i)), \text
{ with } \sum _{i} \alpha (i)=0.
\end{equation} 

We can actually give a symplectic generalization of this notion as follows: 
Let $P _{E} \to \mathbb{CP} ^{n}$ denote the projectivization of $E$, which is a
Hamiltonian bundle and so its total space has a natural deformation class of
symplectic forms $\Omega$, extending the fiber wise symplectic forms $\omega
_{st}$, see \cite{MS2}.

For a smooth rational curve $C$ in $ \mathbb{CP} ^{n}$, we have a smooth
identification of $P _{E}| _{C}$ with $X=\mathbb{CP} ^{r-1} \times S ^{2}$, and
a canonical identification of $H _{*} (X)$ with $H _{*} ( \mathbb{CP} ^{r-1}
\times S ^{2})$, since the group of bundle automorphisms of $X$ acts
trivially on homology by \cite[Theorem 1.16]{symp.str.fiber}.
 \begin{definition} Let $E \to \mathbb{CP} ^{n}$ be a complex vector bundle with
$c _{1} (E)=0$. Let $J$ be an almost complex structure compatible with $\Omega$
on $P _{E}$, and such that the projection map $P _{E} \to \mathbb{CP} ^{n}$ is
$J$-holomorphic. We will call such $J$ \emph {\textbf{admissible}}. A smooth
rational curve $C$ in $ \mathbb{CP} ^{n}$ is called a \emph { \textbf{jumping
curve}}  if $P _{E}| _{C} \simeq \mathbb{CP} ^{r-1} \times S ^{2}$ has a
$J$-holomorphic section in class $d [line] + [S^2]$, $d < 0$. 
\end{definition} 
A more natural way of stating this, is that $P _{E}| _{C}$ has a $J$ holomorphic
section $u$, with $ \langle [\widetilde{\Omega}],  [u] \rangle = d$, $d <0$,
where $[ \widetilde{\Omega}]$ is the coupling class of $P _{E}$, see \cite{MS2}.

When $J$ is induced by a holomorphic structure on $E$, this  notion is
equivalent to the classical notion, since in this case $P _{E}|_C$ is known to
be the generalized Hirzebruch bundle:
\begin{equation*} P _{E}| _{C} \simeq S ^{3} \times _{S ^{1}} \mathbb{CP}
^{r-1},
\end{equation*}
for some circle subgroup $S ^{1} \in SU (r)$. If this subgroup
is non-trivial and $H$ denotes its generating Hamiltonian, then points $x \in F
_{\max}$, (the maximum set of $H$), are fixed points of the $S ^{1}$-action on $
\mathbb{CP} ^{r-1}$, and give holomorphic sections $S ^{3} \times _{S ^{1}}
\{x\}$ of $P _{E}| _{C}$ with above $d <0$.

If we restrict to lines in $ \mathbb{CP} ^{n}$, then 
the locus  of classical jumping lines is a divisor of the variety of all lines
in $ \mathbb{CP} ^{n}$, and if $E$ is holomorphically non-trivial it is
non-empty see  \cite[Section 3]{bundles}. 
We will prove the following partial symplectic generalization of this:
\begin{theorem} \label{thm.jumpline} Let $E \to \mathbb{CP} ^{n}$ be a rank $r \geq n$, 
complex vector bundle, with $c_1 (E)=0$ and some other Chern class non-zero.
Then for any admissible $J$  on $P _{E}$ it has  degree one  jumping
lines. If we also assume that $J$ is suitably generic then $d$ above can be
chosen to be -1. 
\end{theorem}
\subsection*{Acknowledgements} I am deeply grateful to Leonid Polterovich for
explaining to me his ideas on $\karea$, and compelling me to think about this
subject. I also thank Jarek Kedra and
Dusa McDuff and Leonid for discussions and comments, as well as the anonymous
referee for some interesting questions, and numerous comments. This paper was
primarily written during the author's stay at MSRI during spring of 2010, much thanks to the organizers and
administration for creating a pleasant atmosphere. 
 \section {Setup} \label{section.setup}
Here are our main conventions. The Hamiltonian vector field
 generated by $H: (M, \omega) \to \mathbb{R}$ is given by 
 \begin{equation*}  \omega (X _{H}, \cdot) = - dH (\cdot).
\end{equation*}
 And an $\omega$-compatible almost complex structure $J$ is required to satisfy
 $\omega (v, J v) >0$, for $v \neq 0$.  Homology is always over $ \mathbb{Q}$
 unless specified otherwise.
 
 Our main tools are certain characteristic
 cohomology classes
 \begin{equation} \label {eq.qclasses.unstable}
 qc _{k} \in H ^{2k} (\Omega SU (r),
QH ( \mathbb{CP} ^{r-1})),
\end{equation}
where $QH ( \mathbb{CP} ^{r-1})$ denotes the ungraded vector space $H (
\mathbb{CP} ^{r-1}, \mathbb{Q})$, with its ungraded quantum product. These
classes were originally defined and named \emph{quantum classes} in much more generality in \cite{GS}.  Note however, that
$qc_k$ here is $qc _{2k}$ in \cite{GS}. The grading change is for convenience,
as in this context all odd quantum classes vanish.
We will not need the full definition, just some basic geometric content: and the
following theorem \cite{BP}:
   \begin{theorem} \label{theorem.corollary}
The classes $qc  _{k}$ on $\Omega SU (r)$ are algebraically independent and
generate cohomology in the stable range $2k \leq 2r-2$, with coefficients in
$QH ( \mathbb{CP} ^{r-1})$.
\end{theorem}

Here is a brief overview of the geometric construction of quantum
classes. For more details of the following discussion see \cite{GS}. Let
$M \hookrightarrow P \to X$ be a Hamiltonian bundle, with a Hamiltonian
connection $ \mathcal {A}$,
with monotone fiber $(M,
\omega)$, over a smooth manifold $X$. (Here $M$ is used to denote a symplectic
manifold because it serves a different logical purpose to $(X, \omega)$ of
the Introduction, but this $M$ will just be $ \mathbb{CP} ^{r-1}$ in the rest of
the paper.) We have a natural $S ^{1}$-action on $\Omega ^{2} X$, induced by the
rotation of  $S ^{2}$, along the axis of revolution containing the base point $0 \in S ^{2}$. 

Let
$\Omega ^{2} X _{S ^{1}}$ denote the Borel $S ^{1}$-quotient:
\begin{equation*} \Omega ^{2} X _{S ^{1}} = \Omega ^{2} X \times _{S ^{1}} S
^{\infty}. \end{equation*} 

Let $B$ denote a closed, oriented smooth manifold. Given
a cycle $$B \xrightarrow{f} \Omega ^{2} X _{S ^{1}},$$
%
 there is a naturally
induced Hamiltonian bundle $M \hookrightarrow P _{f} \to Y$, where $Y \to B$ is
an oriented $S ^{2}$-bundle over $B$, classified by the composition $B \to
\Omega ^{2} X _{S ^{1}} \to \mathbb{CP} ^{\infty}$, with the map to $ \mathbb{CP} ^{\infty}$ being the canonical projection. Let us explain this: a  map $B \to 
\Omega ^{2} X _{S ^{1}}$ induces an $S ^{1}$-equivariant 
map $T \to X \times S ^{\infty}$, for $T$ an
oriented circle bundle over $B$. (Just by pulling back the universal circle bundle.) But this is
the same thing as an $ S ^{1}$-equivariant map $$T \times S ^{2} \to X \times S
^{\infty},$$ which then induces the $S ^{1}$-quotient map $Y \to X \times \mathbb{CP} ^{\infty}$, for $Y$ an oriented
$S ^{2}$-bundle over $B$. Our $M$ bundle over $ Y$ is then just the pull-back by
the induced map $Y \to X$.  Equivalently we have a bundle 
\begin{equation} \label {eqF} F \hookrightarrow P _{f}
\xrightarrow{p} B,
\end{equation}
 with $p$ denoting natural projection, where $F$ is a
Hamiltonian $M$ bundle over $S ^{2}$.

We may  define  classes $$qc_* \in H ^{*} (B, QH(M))$$ for $P _{f}$ as in
\cite{GS}, via count of certain $p$-fiberwise holomorphic curves, with a
$p$-fiberwise family of complex structures on $P _{f}$, induced by some
Hamiltonian connection $ \mathcal {A}$ on $M \hookrightarrow P \to X$. These
classes are induced by universal classes
\begin{equation}  \label {eq.qclasses} qc_* \in H ^{*} (\Omega ^{2} B\text
{Ham}(M, \omega)_ {S ^{1}}, QH(M)).
\end{equation}

Here are more details in the case $M= \mathbb{CP} ^{r-1}$ and $ \mathbb{CP}
^{r-1} \hookrightarrow P \to X$ is a projectivization of a rank $r$ complex
vector bundle $E$ with $c _{1} (E)=0$. For $f: B \to \Omega ^{2} B \text {Ham}(
\mathbb{CP} ^{r-1}, \omega) _{S ^{1}}$ as above, the fibers $F_b$ of $P _{f}
\to B$ (as in \eqref{eqF}) are
Hamiltonian bundle diffeomorphic to $F= \mathbb{CP} ^{r-1} \times  S ^{2}$,
although not naturally. The group of $ \text {Ham}( \mathbb{CP} ^{r-1})$-bundle
automorphisms of $F_b$ acts trivially on  homology, this follows by 
\cite[Theorem 1.16]{symp.str.fiber}. 
In particular a section class $A$ in $H_2 (F _{b})$ is uniquely characterized by
``degree'' $d$, \begin{equation*} A= d[line] +[\mathbb{CP}^1].
\end{equation*} 

Since $Y \to B$ has a pair of canonical sections  corresponding to the pair of
fixed points of $S ^{1}$ action on $S ^{2}$, we have a pair of natural
embeddings $I: B \times \mathbb{CP} ^{r-1} \to P _{f}$. 

 The classes we now define
``measure'' quantum self intersection of $I(B \times  \mathbb{CP} ^{n}) \subset P
_{f}$.  Let $$\mathcal {M} (P _{f}, d, \{J _{b}\})$$ denote the moduli space 
  of tuples $ (u, b)$, $u$ is a $J _{b}$-holomorphic section of $X
_{b}$ in degree $d$. The virtual dimension of this space is given by the
Fredholm index: \begin{equation*} \label {eq.dim} 2n + 2k + 2 \langle c_1
^{vert}, {A} \rangle = 2n + 2k + 2d \cdot (n+1) .
\end{equation*}

We define $qc _{k} \in H ^{2k} (\Omega \text {Ham}( \mathbb{CP} ^{n}, \omega),
QH ( \mathbb{CP} ^{n}))$ as follows: \begin{equation} \label {eq.def.ch.class}
\langle qc_k, [f] \rangle = \sum_ {d \in \mathbb{Z}} b _{d}. \end{equation} 
Where $b _{d} \in H_* ( \mathbb{CP} ^{n})$ is defined by
  duality: \begin{equation*} b _{d} \cdot _{ \mathbb{CP} ^{n}} c = ev _{d}
  \cdot _{B \times \mathbb{CP} ^{n}} [B] \otimes c, 
 \end{equation*} and where  
\begin{align*} ev _{d}: \mathcal {M}  (P _{f}, d, \{J _{b}\}) \to B
\times \mathbb{CP} ^{n} \\ ev _{d} (b, u) = (b, u (0)),
\end{align*}
 and $\cdot _{M}, \cdot _{B \times M}$ denote the 
intersection pairings in $ M$, respectively $ B \times M$. 
The sum  \eqref{eq.def.ch.class}, is  finite and only $d<0$ contribute for
dimensional reasons.
\section {Proofs}

\begin{proof}[Proof of Theorem \ref{thm.jumpline}] Let $E$ be a rank $r$
complex vector bundle over $ \mathbb{CP} ^{n}$, with $n \leq r$, and some Chern
class non-zero. We may assume without loss of generality that $E$ has a
non-vanishing Chern number. (Otherwise,  restrict the following discussion to a
subspace $ \mathbb{CP} ^{i} \subset \mathbb{CP} ^{n}$, corresponding to a
non-zero class $c_i (E)$.) And let $$ \mathcal {M} _{0,1 }(\mathbb{CP} ^{n},
[line],j; x_0, [\mathbb{CP} ^{n-1}]) \to \mathbb{CP} ^{n},$$ denote the moduli
 space of curves with 1 free marked point and 2 fixed marked points mapping to
  $x_0$, $ \mathbb{CP} ^{n-1}$, $x_0 \notin \mathbb{CP} ^{n-1}$, with
   $$ev: \mathcal {M} _{0,1 }(\mathbb{CP} ^{n}, [line],j; x_0, [\mathbb{CP} ^{n-1}]) \to \mathbb{CP} ^{n}, $$ 
denoting the evaluation map given by evaluating at the free marked point. It is
well known that the standard complex structure on $ \mathbb{CP} ^{n}$ is regular
and that for this standard $j$, the evaluation map  is a degree one map
\begin{equation*} P \to \mathbb{CP} ^{n},
\end{equation*}
where $P$ is an $S ^{2}$-bundle over $ \mathbb{CP} ^{n-1}$ associated to the
Hopf bundle. This is because there is a unique complex line through a pair of
points in $ \mathbb{CP} ^{n}$. The induced cycle $f: \mathbb{CP} ^{n-1} \to
\Omega ^{2} \mathbb{CP} ^{n} _{ S ^{1}}$ represents a class denoted $a$.
Let $e: \mathbb{CP}
^{\infty} \to \Omega ^{2} BSU (r)_{S ^{1}} = \Omega ^{2} BSU (r) \times _{S ^{1}}
S ^{\infty}$ be the  section corresponding to the canonical fixed point of the $S
^{1}$-action on $\Omega ^{2} BSU (r)$, i.e. the constant map of $S ^{2}$ to the
based point $x_0 \in X$.
 \begin{lemma} \label{lemma.k-essential}   $$0 \neq f _{  {E}_*} a\in H _{2n-2} (\Omega ^{2}BSU (r) _{S ^{1}})/e_* H_* ( \mathbb{CP}
^{\infty}),$$ where $$f _{E}: \Omega ^{2} \mathbb{CP} ^{n}_{S ^{1}} \to \Omega
^{2}BSU (r)_{S ^{1}}$$ is the map induced by $ E \to \mathbb{CP} ^{n}$.
\end{lemma}
\begin {proof}
Let us suppose otherwise. 
%
The composition map $$ev: P \to \mathbb{CP} ^{n} \to
BSU (r),$$ is non vanishing in homology, since $E$ has a non-vanishing Chern
number and $$ev: P  \to \mathbb{CP} ^{n}$$ is degree one by discussion above.

 Let $$H: T \to \Omega ^{2} BSU (r)_{S ^{1}}$$ be a bordism of
$f _{ E _{*}}a _{i-1}$ to $c \in e_*H_* ( \mathbb{CP} ^{\infty})$. We'll call
the corresponding boundary pieces of $T$ by $T _{a}$ and $T _{c}$.  Consequently, 
the bordism $H$ induces an $S ^{2}$-bundle
$P _{T}$ over $T$,  it is the pull-back of the tautological $S
 ^{2}$-bundle $$(\Omega ^{2} _{x_0}BSU (r) \times E ^{\infty}) \times _{S ^{1}}
 S ^{2} \to\Omega ^{2} _{x_0} BSU (r)_{S ^{1}}.$$ and of course  $P _{T}$
restricts over $T _{a}$ to $P$. We have a natural ``evaluation'' map
\begin{equation*} ev _{T}: P _{T} \to BSU (r),
\end{equation*} 
restricting to evaluation maps $ev$, $ev_c$ over boundary, and so a homology
of $[ev]$ to $[ev_c]$, but $ev_c$ is the constant map to the based point $x_0
\in BU$ a contradiction.
\end {proof}
\begin{lemma} \label{lemma.nonvanish}  For some integer $\{\alpha_i, \beta
_{i}\}$ with at least one $\alpha _{i}, \beta _{i}$ non-zero:
 \begin{equation} \label {eq.nonvanish2} \langle \prod _{i,j} qc ^{\alpha_i}
 _{\beta_i} \wedge c_1 ^{j}, f _{{ E} _*} a _{n-1} \rangle \neq 0, 
\end{equation}
where $c_1$ denotes the pullback to $\Omega ^{2} BSU (r) _{S ^{1}}$ of the
canonical generator of $H ^{2} ( \mathbb{CP} ^{\infty})$ by the natural
projection $\Omega ^{2} BSU (r) _{S ^{1}} \to \mathbb{CP} ^{\infty}$.
\end{lemma} 

\begin{proof} Note that all of the rational cohomology of $\Omega ^{2} BSU (r)
\simeq \Omega SU (r)$, is in even degree, since by Milnor-Morre, Cartan-Serre \cite{MM},
\cite{CS}, the rational homology algebra of $\Omega SU (r)$ is generated as a
ring with Pontryagin product by the  rational homotopy groups, (via Hurewicz homomorphism) which are
all in even degrees since the rational homotopy groups of $SU (r)$ are well known to be all in odd degrees. (In fact $SU (r)$ has the
rational homotopy type of the product of odd spheres $S ^{3} \times S ^{5}
\ldots$) Consequently,  the Serre spectral sequence for the fibration $$\Omega
^{2} BSU (r) \hookrightarrow \Omega ^{2} BSU (r) _{S ^{1}} \to \mathbb{CP} ^{\infty}$$
degenerates at the second page and so: \begin{equation*} H ^{*} (\Omega ^{2} BSU (r) _{S ^{1}}) \simeq H ^{*} (\Omega
 ^{2} BSU (r)) \otimes H ^{*} ( \mathbb{CP} ^{\infty}) \simeq H ^{*} (\Omega SU (r))
\otimes H ^{*} ( \mathbb{CP} ^{\infty}).
\end{equation*}
 Our lemma then follows by  Lemma \ref{lemma.k-essential} and Theorem
 \ref{theorem.corollary}. 
 \end{proof}
 The theorem then readily follows. Since by  construction of quantum
 classes and
 Lemma \ref{lemma.nonvanish},  for any fixed (not necessarily regular) complex
 structure $J$ on $P$ compatible with $\Omega$, and with projection to $
 \mathbb{CP} ^{n}$ for some complex line $l$ in $
 \mathbb{CP} ^{n}$, the restriction of $P$ to $l$, which is diffeomorphic to $ \mathbb{CP} ^{r-1}
 \times S ^{2}$ has a $J $ holomorphic stable section $u$ in total class $S= -
 [line] + S ^{2}$. As otherwise the relevant Gromov-Witten invariants in class
 $S$ all vanish and 
 \eqref{eq.nonvanish2} is impossible. Of course the stable
 section $u$ may be in the form of a holomorphic section $u _{p}$ in class $
 d[line] + S ^{2}$, with $d<-1$ together with some vertical holomorphic
 bubbles, but this still implies our claim.
\end {proof}
\begin {proof} [Proof of Theorem \ref{thm.projective}]
We just need the following lemma:
\begin{lemma}  \label{lemma.norm} The norm of the curvature $||R _{ \mathcal
{A}}||$ of the projectivization $ \mathbb{CP} ^{r-1} \hookrightarrow P \to
\mathbb{CP} ^{n}$ is at least $1$. 
\end{lemma}
\begin{proof} 
 Let $ \widetilde{\Omega}$ denote the coupling form of the Hamiltonian fibration
 $P$ associated to $ \mathcal {A}$, (see for example \cite{MS2} for discussion on coupling forms). This  is a
certain closed form associated to  the curvature form of $ \mathcal {A}$, 
with the following properties.

The 
restriction of $ \widetilde{\Omega} $ to fibers $M \simeq \mathbb{CP} ^{r-1}$ of
$P \to \mathbb{CP} ^{n}$ coincides with $\omega _{st}$. The $
\widetilde{\Omega}$-orthogonal subspaces in $TP$ to the fibers are horizontal
subspaces, whose value on horizontal lifts $$ \widetilde{v}, \widetilde{w} \in T _{m,z} P$$ of $v,w \in T _{z} \mathbb{CP} ^{n}$
are given by $$ \widetilde{\Omega} ( \widetilde{v}, \widetilde{w}) = -R _{ \mathcal {A}} (v, w) (m),$$ 
for $m \in M _{z}$; in other words we evaluate the Lie algebra element of $ \text
{Ham}(M_z, \omega)$: $R _{ \mathcal {A}} (v, w)$ (i.e. a function on $M _{z}$)
at $m$.

Consider the symplectic form $\Omega =  \widetilde{\Omega}
+ (||R ^{ \mathcal {A}}|| + \epsilon) \omega _{st} $, where $\omega _{st}
$, the standard Fubini-Study symplectic form on the base normalized by the
condition that the area of a complex line is 1, and $\epsilon>0$. Pick any
compatible complex structure $J _{ \mathcal {A}}$. By Theorem
\ref{thm.jumpline}, for some complex line $l$ in $ \mathbb{CP} ^{n}$,
the restrictions of $P$ to $l$, which is diffeomorphic to $
\mathbb{CP} ^{r-1} \times S ^{2}$ (by assumption that $c _{1} (E)=0$) has a $J
_{ \mathcal {A}}$-holomorphic section $u$ in class $S= d \cdot [line] + S ^{2}$, with $d \leq -1$. Where $[line]$ is
the class of the complex line in $ \mathbb{CP} ^{r-1}$.

Since $J _{ \mathcal {A}}$ is $\Omega$ compatible, for the
class $S$, $J _{ \mathcal {A}}$-holomorphic section $u$ of $P|_l$ we get
\begin{equation*} 0 \leq [\Omega] ( [u])  = [ \widetilde{\Omega }] ( [u])
+   ||R ^{ \mathcal {A}}|| + \epsilon.
\end{equation*}
On the other hand $[\widetilde{\Omega}] = [\omega _{st}]$ on $P| _{l}$ since the
cohomology class of the coupling form is independent of the choice of
connection, and the form $\omega _{st}$ on $ \mathbb{CP} ^{r-1} \times S ^{2}$,
is another coupling form associated to the trivial connection on this bundle.
Since $ [ \omega _{st}] ( [line])=1$ by our normalization, it follows that $ [ \widetilde{\Omega} ] ([u])= d$, since $[u] =
d[line]+ S ^{2}$. So we get: \begin{equation*} -d \leq  ||R ^{ \mathcal {A}}|| +
\epsilon,
\end{equation*}
for every $\epsilon >0$. 
\end{proof}
This finishes the proof of the theorem.
\end{proof}
\bibliographystyle{siam}  
\bibliography{link} 
\end{document}